\documentclass[final,reqno]{siamltex}
\usepackage{latexsym,amsmath,amssymb,amsfonts,mathrsfs}
\usepackage{epsf,graphicx,epsfig,color,cite,cases}
\usepackage{subfigure,graphics,multirow,marginnote,graphicx}
\newcommand{\R}{{\mathbb R}}
\newcommand{\Z}{{\mathbb Z}}
\newcommand{\N}{{\mathbb N}}
\newcommand{\C}{{\mathbb C}}
\newcommand{\Sp}{{\mathbb S}}

\newcommand{\be}{\begin{eqnarray}}
\newcommand{\ben}{\begin{eqnarray*}}
\newcommand{\en}{\end{eqnarray}}
\newcommand{\enn}{\end{eqnarray*}}
\newcommand{\ba}{\backslash}
\newcommand{\pa}{\partial}

\newcommand{\ov}{\overline}

\newcommand{\I}{{\rm Im}}
\newcommand{\Rt}{{\rm Re}}

\newcommand{\la}{\lambda}

\newcommand{\wid}{\widetilde}

\newcommand{\se}{\setminus}

\newtheorem{remark}[theorem]{Remark}

\begin{document}
\renewcommand{\theequation}{\arabic{section}.\arabic{equation}}
\title{\bf Uniqueness in inverse scattering problems with phaseless far-field data at a fixed frequency
}
\author{Xiaoxu Xu\thanks{Academy of Mathematics and Systems Science, Chinese Academy of Sciences,
Beijing 100190, China and School of Mathematical Sciences, University of Chinese
Academy of Sciences, Beijing 100049, China ({\tt xuxiaoxu14@mails.ucas.ac.cn})}
\and Bo Zhang\thanks{NCMIS, LSEC and Academy of Mathematics and Systems Science, Chinese Academy of
Sciences, Beijing, 100190, China and School of Mathematical Sciences, University of Chinese
Academy of Sciences, Beijing 100049, China ({\tt b.zhang@amt.ac.cn})}
\and Haiwen Zhang\thanks{Academy of Mathematics and Systems Science, Chinese Academy of Sciences,
Beijing 100190, China ({\tt zhanghaiwen@amss.ac.cn})}}
\date{}

\maketitle

\begin{abstract}
This paper is concerned with uniqueness in inverse acoustic scattering with phaseless far-field data
at a fixed frequency. The main difficulty of this problem is the so-called translation invariance
property of the modulus of the far-field pattern generated by one plane wave as the incident field.
Based on our previous work (J. Comput. Phys. 345 (2017), 58-73), the translation invariance property
of the phaseless far-field pattern can be broken by using infinitely many sets of superpositions of
two plane waves as the incident fields at a fixed frequency.
In this paper, we prove that the obstacle and the index of refraction of an inhomogeneous
medium can be uniquely determined by the phaseless far-field patterns generated by infinitely many sets of
superpositions of two plane waves with different directions at a fixed frequency under the condition
that the obstacle is a priori known to be a sound-soft or non-absorbing impedance obstacle and the
index of refraction $n$ of the inhomogeneous medium is real-valued and satisfies that either $n-1\ge c_1$
or $n-1\le-c_1$ in the support of $n-1$ for some positive constant $c_1$.
To the best of our knowledge, this is the first uniqueness result in inverse scattering with phaseless
far-field data. Our proofs are based essentially on the limit of the normalized eigenvalues of the far-field
operators which is also established in this paper by using a factorization of the far-field operators.
\end{abstract}



\pagestyle{myheadings}
\thispagestyle{plain}
\markboth{X. Xu, B. Zhang, and H. Zhang}{Uniqueness in inverse scattering with phaseless far-field data}

\section{Introduction}\label{sec1}

Inverse scattering by bounded obstacles and inhomogeneous media has wide applications in many areas
such as radar and sonar detection, geophysical prospection, medical imaging and nondestructive testing
(see, e.g. \cite{CK}). The present paper is concerned with inverse scattering with phaseless far-field
data associated with incident plane waves.

To provide a precise description of the problem, assume that the obstacle $D$ is an open and bounded
domain in $\R^3$ with $C^2$-boundary $\pa D$ such that the exterior $\R^3\se\ov{D}$ of $\ov{D}$ is connected.
Consider the time-harmonic ($e^{-i\omega t}$ time dependence) plane wave
\ben
u^i=u^i(x,d):=\exp(ikd\cdot x)
\enn
which is incident on the obstacle $D$ from the unbounded part $\R^3\ba\ov{D}$,
where $d\in\Sp^2$ is the incident direction with $\Sp^2$ denoting the unit sphere in $\R^3$, $k=\omega/c>0$
is the wave number, $\omega$ and $c$ are the wave frequency and speed in the homogeneous medium in
$\R^3\ba\ov{D}$. Denote by $u^s$ the scattered field. Then the total field $u:=u^i+u^s$ outside
an impenetrable obstacle $D$ satisfies the exterior boundary value problem:
\begin{subequations}
\be\label{he}
\Delta u+k^2 u&=&0\quad{\rm in}\;\;\R^3\se\ov{D},\\ \label{bc}
{\mathscr B}u&=&0\quad{\rm on}\;\;\pa D,\\ \label{rc}
\lim\limits_{r\rightarrow\infty}r\left(\frac{\pa u^s}{\pa r}-iku^s\right)&=&0,\quad r=|x|.
\en
\end{subequations}
Here, (\ref{he}) is the well-known Helmholtz equation, (\ref{rc}) is the Sommerfeld
radiation condition which ensures the uniqueness of the scattered field $u^s$.
The boundary condition ${\mathscr B}$ in (\ref{bc}) depends on the physical property of the obstacle $D$:
\ben
\begin{cases}
{\mathscr B}u=u & \text{for a sound-soft obstacle},\\
{\mathscr B}u={\pa u}/{\pa\nu}+\eta u & \text{for an impedance obstacle},
\end{cases}
\enn
where $\nu$ is the unit outward normal to the boundary $\pa D$ and $\eta$ is
the impedance function satisfying that $\I[\eta(x)]\geq0$ for all $x\in\pa D$.
In this paper, we assume that $\eta\in C(\pa D)$, that is, $\eta$ is continuous on $\pa D$.
If $\I[\eta(x)]>0$ for all $x\in\pa D$, we say the impedance boundary condition is
\emph{absorbing}; if $\I[\eta(x)]=0$ for all $x\in\pa D$, we say the impedance boundary
condition is \emph{non-absorbing}. When $\eta=0$, the impedance boundary condition becomes
the Neumann boundary condition (sound-hard obstacles), so the Neumann boundary condition is non-absorbing.

The problem of scattering of a plane wave by an inhomogeneous medium is modelled
by the medium scattering problem:
\begin{subequations}
\be\label{he-n}
\Delta u+k^2n(x)u&=&0\quad\mbox{in}\;\;\R^3,\\ \label{rc-n}
\lim\limits_{r\rightarrow\infty}r\left(\frac{\pa u^s}{\pa r}-iku^s\right)&=&0,\quad r=|x|,
\en
\end{subequations}
where $u=u^i+u^s$ is the total field, $u^s$ is the scattered field, and $n$ in the reduced wave
equation (\ref{he-n}) is the refractive index characterizing the inhomogeneous medium.
In this paper, we assume that $m:=n-1$ has compact support $\ov{D}$ and $n\in L^\infty(D)$ with
$\Rt[n(x)]>0$ and $\I[n(x)]\geq0$ for all $x\in D$. If $\I[n(x)]>0$ for all $x\in D$,
then the medium is called \emph{absorbing}; if $\I[n(x)]=0$ for all $x\in D$ (so $n$ is real-valued),
then the medium is called \emph{non-absorbing}.

The existence of a unique (variational) solution to the problems (\ref{he})-(\ref{rc}) and
(\ref{he-n})-(\ref{rc-n}) has been proved in \cite{CK,KG,Kirsch} (see Theorem 3.11 in \cite{CK}
and Theorem 1.1 in \cite{KG} for the exterior Dirichlet problem, Theorem 2.2 in \cite{KG}
for the exterior impedance problem and Theorem 6.9 in \cite{Kirsch} for the medium scattering
problem (\ref{he-n})-(\ref{rc-n})). In particular, it is well-known that the scattered field $u^s$
has the asymptotic behavior:
\ben
u^s(x,d)=\frac{e^{ik|x|}}{|x|}\left\{u^\infty(\hat{x},d)+\left(\frac{1}{|x|}\right)\right\},
\quad|x|\rightarrow\infty
\enn
uniformly for all observation directions $\hat{x}=x/|x|\in\Sp^2$, where $u^\infty(\hat{x},d)$ is the
far field pattern of $u^s$ which is an analytic function of $\hat{x}\in\Sp^2$ for each $d\in\Sp^2$
and of $d\in\Sp^2$ for each $\hat{x}\in\Sp^2$ (see, e.g., \cite{CK}).

The inverse scattering problem is to determine the shape and location of the obstacle $D$ and its
physical property or the index of refraction $n$ of the medium from the near-field (the scattered
field $u^s$ or the total field $u$) or the far field pattern $u^\infty$.
In many practical applications, the phase of the near-field or the far-field pattern can not be measured
accurately compared with its modulus or intensity. Thus, it is often desirable to recover the obstacle
and the medium from the modulus or intensity of the near-field or the far-field pattern (or the phaseless
near-field data or the phaseless far-field data).
The inverse scattering problem with phaseless near-field data ($|u^s|$ or $|u|$ on a measurement surface
enclosing the obstacle or the medium) or the phaseless far-field pattern $|u^\infty|$ is called
the \emph{inverse scattering problem with phaseless data},
while the inverse scattering problem with the near field data (the scattered field $u^s$ or the total field
$u$ on a measurement surface enclosing the obstacle or the medium) or the far-field pattern $u^\infty$ is
called the \emph{inverse scattering problem with full data}.

Over the past three decades, inverse scattering problems with full data have been extensively studied
mathematically and numerically (see, e.g., the monographs \cite{CK,KG} and the references quoted there).
The inverse scattering problem with phaseless near-field data is also called the (near-field)
\emph{phase retrieval problem} in optics and other physical and engineering sciences and has also been
widely studied numerically over the past decades (see, e.g. \cite{Candes13,Candes15}).
For examples, many reconstruction algorithms have been developed to recover the obstacle or the refractive
index of the medium from the phaseless near-field data (or the phaseless total near-field data),
corresponding to plane wave incidence or point source incidence
(see, e.g. \cite{TGRS03,CH16,CH17,Li09,MDS92,MD93,Pan11,MOTL97} and the references quoted there).
However, few results are available for the mathematical study (such as uniqueness and stability) of the
near-field phase retrieval problem.
This is mainly because Rellich's lemma (see \cite[Lemma 2.12]{CK}), which not only ensures uniqueness for
scattering solutions and also establishes the one-to-one correspondence between the scattered fields and
their far-field patterns and plays an essential role for the mathematical study of the inverse scattering
problems with full data, dose not work any more for inverse scattering with phaseless data.
This fact makes the phaseless inverse scattering problems much more difficult to study mathematically.
Recently, a uniqueness result has been established in \cite{K14} for
recovering a nonnegative, smooth, compactly supported, real-valued potential from the phaseless near-field data
corresponding to all incident point sources placed on a spherical surface for an interval of frequencies.
This uniqueness result was extended in \cite{K17} to the case of recovering the smooth wave speed in the
three-dimensional Helmholtz equation. Reconstruction procedures were introduced in \cite{KR16,N15,N16}
for the inverse medium scattering problems with phaseless near-field data.
Recently in \cite{MH17}, the stability analysis was established for a linearized near-field phase retrieval
problem for weakly scattering objects known as the contrast transfer function model in X-ray phase contrast
imaging.

In contrast to the case with phaseless near-field data, inverse scattering with phaseless far-field data
is more challenging due to the \emph{translation invariance property} of the phaseless far-field pattern
which was proved in \cite{KR} for sound-soft obstacles and in \cite{LS04} for sound-hard and impedance
obstacles (see also \cite{ZZ01}). For the shifted obstacle $D_\ell:=\{x+\ell:x\in D\}$ or the refractive
index of the shifted inhomogeneous medium, $n_\ell(x):=n(x-\ell)$ with $\ell\in\R^3$, the scattered
field $u_\ell^s$ corresponding to the incident plane wave $u^i(x,d)=e^{ikx\cdot d}$ is given as
\ben
u_\ell^s(x,d)=e^{ik\ell\cdot d}u^s(x-\ell,d),\quad x\in\R^3\setminus\ov{D_\ell},\;d\in\Sp^2
\enn
in terms of the scattered field $u^s$ corresponding to $D$ or $n$. The corresponding far-field pattern is
\ben
u_\ell^\infty(\hat{x},d)=e^{ik\ell\cdot(d-\hat{x})}u^\infty(\hat{x},d),\quad\hat{x},d\in\Sp^2,\;\ell\in\R^3,
\enn
where $u_\ell^\infty$ is the far-field pattern for the obstacle $D_\ell$ or the refractive index $n_\ell$.
Thus we have the translation invariance property
\be\label{TIP}
|u_\ell^\infty(\hat{x},d)|=|u^\infty(\hat{x},d)|,\quad\hat{x},d\in\Sp^2,\;\ell\in\R^3.
\en
This means that it is impossible to reconstruct the location of the obstacle $D$ or the inhomogeneous
medium from the phaseless far-field pattern with one plane wave as the incident field.
However, several reconstruction algorithms have been developed to reconstruct the shape of the obstacle
from the phaseless far-field data with one plane wave as the incident field
(see \cite{I2007,IK2010,IK2011,KR,LiLiu15,S16}).
For plane wave incidence no uniqueness results are available for the general inverse obstacle scattering
problems with phaseless far-field data. By assuming a priori the obstacle to be a sound-soft ball centered
at the origin, uniqueness was established in determining the radius of the ball from a single phaseless
far-field datum in \cite{LZ10}.

We remark that a continuation algorithm was proposed in \cite{BaoLiLv2012} to reconstruct the shape
of a perfectly reflecting grating profile from the phaseless near-field data associated with incident
plane waves and a recursive linearization algorithm in frequencies was introduced in \cite{BaoZhang16}
to recover the shape of multi-scale sound-soft large rough surfaces from phaseless measurements of
the scattered field generated by tapered waves with multiple frequencies.
Note that the phaseless near-field data is also invariant under translations in the non-periodic
direction of the periodic grating profile in the former case and in the vertical direction of the
unbounded rough surface in the later case.

As discussed above, for plane wave incidence it is the translation invariance property (\ref{TIP})
which makes it impossible to recover the location of the scattering obstacle from phaseless far-field data.
Recently in \cite{ZZ01}, it was proved that the translation invariance property of the phaseless far-field
pattern can be broken if superpositions of two plane waves rather than one plane wave are used as the
incident fields with an interval of frequencies. A recursive Newton-type iteration algorithm in frequencies
was also developed in \cite{ZZ01} to recover both the location and the shape of the obstacle simultaneously
from multi-frequency phaseless far-field data. This approach was further extended to inverse scattering by
locally rough surfaces with phaseless far-field data in \cite{ZZ02}.
On the other hand, by means of the results in \cite{ZZ01} it is also easy to see that the translation
invariance property of the phaseless far-field pattern can be broken by using infinitely many sets of
superpositions of two plane waves with different directions as the incident fields at a fixed frequency.
Based on this, we have recently developed a fast imaging algorithm in \cite{ZZ03} to recover scattering
obstacles by phaseless (or intensity-only) far-field data at a fixed frequency.

Motivated by \cite{ZZ01,ZZ02,ZZ03}, we prove in this paper that the obstacle $D$ and the refractive index
$n$ can be uniquely determined by the phaseless far-field patterns generated by infinitely many sets of
superpositions of two plane waves with different directions at a fixed frequency under the condition
that $D$ is a priori known to be a sound-soft or a non-absorbing impedance obstacle and $n$ is a priori
assumed to be real-valued with the condition that either $n-1\ge c_1$ or $n-1\le-c_1$ for some constant
$c_1>0$. As far as we know, this is the first uniqueness result in inverse scattering
with phaseless far-field data. Our proofs are based essentially on the limit of the normalized eigenvalues
of the far-field operators (Theorem \ref{eigentend}) which is established in this paper by using
a factorization of the far-field operators. It should be pointed out that the limit of the normalized
eigenvalues of the far-field operators is established only for sound-soft or non-absorbing impedance
obstacles and for non-absorbing inhomogeneous media. For the absorbing cases, however, the limit of the
normalized eigenvalues of the far-field operators is not known yet.

This paper is organized as follows. The main results are presented in Section \ref{main-results}.
Spectral properties of the far-field operators are established in Section \ref{spectrum}.
Section \ref{sec4} is devoted to the proof of the main results.
Conclusions are given in Section \ref{conclusion}.

\section{The main results}\label{main-results}

In this section, let the wave number $k$ be arbitrarily fixed. Following \cite{ZZ01,ZZ03},
we make use of the following superposition of two plane waves as the incident field:
\ben
u^i=u^i(x;d_1,d_2)=u^i(x,d_1)+u^i(x,d_2)=e^{ikx\cdot d_1}+e^{ikx\cdot d_2},
\enn
where $d_1,d_2\in\Sp^2$ are the incident directions with $d_1\not=d_2$.
Then the scattered field $u^s$ has the asymptotic behavior
\ben
u^s(x;d_1,d_2)=\frac{e^{ik|x|}}{|x|}\left\{u^\infty(\hat{x};d_1,d_2)
+O\left(\frac{1}{|x|}\right)\right\},\;\;|x|\to\infty,
\enn
uniformly for all observation directions $\hat{x}\in\Sp^2$. From the linear superposition principle
it follows that \[u^s(x;d_1,d_2)=u^s(x,d_1)+u^s(x,d_2)\] and
\be\label{lsp}
u^\infty(\hat{x};d_1,d_2)=u^\infty(\hat{x},d_1)+u^\infty(\hat{x},d_2),
\en
where $u^s(x,d_j)$ and $u^\infty(\hat{x},d_j)$ are the scattered field and its far-field pattern
corresponding to the incident plane wave $u^i(x,d_j)$, $j=1,2.$

Denote by $u_j^s$ and $u_j^\infty$ the scattered field and its far-field pattern, respectively,
associated with the obstacle $D_j$ (or the inhomogeneous medium with the refractive index $n_j$) and
corresponding to the incident field $u^i$, $j=1,2$.
Then we have the following main results on uniqueness in inverse scattering with phaseless far-field data.

\begin{theorem}\label{main}
$(i)$ Assume that $D_1$ and $D_2$ are two sound-soft obstacles. If the corresponding far-field patterns
coincide, that is,
\be\label{m=}
|u_1^\infty(\hat{x};d_1,d_2)|=|u_2^\infty(\hat{x};d_1,d_2)|,\quad\forall\hat{x},d_1,d_2\in\Sp^2,
\en
then $D_1=D_2$.

$(ii)$ Assume that $D_1$ and $D_2$ are two non-absorbing impedance obstacles (i.e., $\I(\eta_j)=0$,
$j=1,2$) with the impedance coefficient $\eta_j\in C(\pa D_j)$, $j=1,2$.
If the corresponding far-field patterns satisfy $(\ref{m=})$, then $D_1=D_2$ and $\eta_1=\eta_2$.

$(iii)$ Assume that $n_1,n_2\in L^\infty(\R^3)$ are the refractive indices of two non-absorbing inhomogeneous
media (i.e., $\I(n_j)=0$, $j=1,2$) with $n_j-1$ supported in $\ov{D_j}$, $j=1,2$. Assume further that there is
a constant $c_1>0$ such that either $n_j-1\ge c_1$ in $D_j$ ($j=1,2$) or $n_j-1\le-c_1$ in $D_j$ ($j=1,2$).
If the corresponding far field patterns satisfy $(\ref{m=})$, then $n_1=n_2$.
\end{theorem}

We need much less phaseless far-field data to determine sound-soft or non-absorbing impedance obstacles, as
seen from the following theorem.

\begin{theorem}\label{main2}
$(i)$ Assume that $D_1$ and $D_2$ are two sound-soft obstacles. If the corresponding far-field patterns
satisfy that
\be\label{r=}
|u_1^\infty(\hat{x},d)|=|u_2^\infty(\hat{x},d)|,\quad\forall\hat{x},d\in\Sp^2
\en
and
\be\label{m=2}
|u_1^\infty(\hat{x};d,d_0)|=|u_2^\infty(\hat{x};d,d_0)|,\quad\forall \hat{x},d\in\Sp^2
\en
for an arbitrarily fixed $d_0\in\Sp^2$, then $D_1=D_2$.

$(ii)$ Assume that $D_1$ and $D_2$ are two non-absorbing impedance obstacles (i.e., $\I(\eta_j)=0$,
$j=1,2$) with the impedance coefficient $\eta_j\in C(\pa D_j)$, $j=1,2$.
If the corresponding far-field patterns satisfy $(\ref{r=})$ and $(\ref{m=2})$, then $D_1=D_2$ and
$\eta_1=\eta_2$.
\end{theorem}

\begin{remark}\label{re1} {\rm
(i) Theorem \ref{main2} remains true for the two-dimensional case, and the proof is the same.

(ii) Theorem \ref{main} (ii) also holds in two dimensions if the assumption $n_1,n_2\in L^\infty(\R^3)$
is replaced by the condition that $n_1,n_2$ are piecewise in $W^{1,p}(D)$ with $p>2$. In this case,
the proof is similar except that we need Bukhgeim's result in \cite{B} (see also the theorem in
Section 4.1 in \cite{Blasten}) instead of \cite[Theorem 6.26]{Kirsch} in the proof.
}
\end{remark}

\section{Spectrum of the far-field operators}\label{spectrum}

Our proofs of the main Theorems \ref{main} and \ref{main2} depend essentially on the spectral properties
of the far-field operator $F:L^2(\Sp^2)\rightarrow L^2(\Sp^2)$ defined by
\be\label{F}
(Fg)(\hat{x}):=\int_{\Sp^2}u^\infty(\hat{x},d)g(d)ds(d),\quad \hat{x}\in\Sp^2.
\en
Since the kernel $u^\infty(\hat{x},d)$ is analytic in $\hat{x}\in\Sp^2$ and in $d\in\Sp^2$, respectively,
the far-field operator $F$ is compact and hence has at most a countable number of eigenvalues with the only
possible accumulation point being $0$. Further, the multiplicity of each eigenvalue is finite since the
dimension of the null space of $\lambda I-F$ for each non-zero eigenvalue $\lambda$ is finite.
We need more spectral properties of the far-field operators.

\subsection{A countably infinite number of eigenvalues of far-field operators}

In this subsection, we will show that $F$ has a countably infinite number of non-zero eigenvalues
for the cases considered in this paper.
To this end, we introduce the adjoint operator $F^*:L^2(\Sp^2)\rightarrow L^2(\Sp^2)$ defined by
\[
(Fg,h)=(g,F^*h),\quad\forall g,h\in L^2({\mathbb S}^2),
\]
where $(\cdot,\cdot)$ is the inner product of $L^2({\mathbb S}^2)$. The following lemma shows that
$F$ is normal for the cases considered in this paper, which was proved in \cite{KG} for a sound-soft or
sound-hard obstacle (see \cite[Theorem 1.8]{KG}) and an impedance obstacle with real-valued impedance
coefficient $\eta$ (see \cite[Theorem 2.5 (e)]{KG}) and in \cite{Kirsch} for an inhomogeneous medium
with real-valued index of refraction $n\in L^\infty(\R^3)$ satisfying that $n-1$ has a compact support
(see \cite[Theorem 6.16]{Kirsch}).

\begin{lemma}\label{normal}
For a sound-soft obstacle or an impedance obstacle with real-valued impedance function
$\eta\in L^\infty(\pa D)$ and for an inhomogeneous medium with real-valued index of refraction
$n\in L^\infty(\R^3)$ satisfying that $n-1$ has a compact support, we have
\be\label{adjoint}
F-F^*-\frac{ik}{2\pi}F^*F=0,
\en
so $F$ is normal.
\end{lemma}

\begin{remark}\label{re-circle} {\rm
Define the \emph{scattering operator} $\mathcal{S}:L^2({\mathbb S}^2)\rightarrow L^2({\mathbb S}^2)$ by
\[\mathcal{S}:=I+\frac{ik}{2\pi}F.\]
Then
\[\mathcal{S}^*\mathcal{S}=(I-\frac{ik}{2\pi}F^*)(I+\frac{ik}{2\pi}F)=I.\]
Hence $\mathcal{S}$ is \emph{unitary}. From the fact that the eigenvalues of a unitary operator lie on the
unit circle in the complex plane, we know that every non-zero eigenvalue of $F$ lies on the circle
\be\label{circle}
\left\{z\in{\mathbb C}:\left|z-\frac{2\pi i}{k}\right|=\frac{2\pi}{k}\right\}
\en
in the upper-half complex plane. We later show (Theorem \ref{eigentend}) that the eigenvalues tend to
zero from the left half of this circle (or the normalized non-zero eigenvalues tend to $-1$) for both
sound-soft obstacles and real-valued indices of refraction $n$ with $n-1\le-c_1<0$ in $D$
and from the right half of this circle (or the normalized non-zero eigenvalues tend to $1$)
for impedance obstacles with real-valued impedance coefficients $\eta$ and real-valued indices of
refraction $n$ with $n-1\ge c_1>0$ in $D$.
Our proof is based essentially on a factorization of the far-field operator in which the middle
operator can be decomposed into the sum of a coercive operator and a compact part.
}
\end{remark}

To obtain a factorization of the far-field operator we introduce the following boundary integral operators
$S,K,K':H^{-1/2}(\pa D)\to H^{1/2}(\pa D)$ and $T:H^{1/2}(\pa D)\to H^{-1/2}(\pa D)$ defined by
\ben
(S\varphi)(x)&:=&\int_{\pa D}\Phi_k(x,y)\varphi(y)ds(y),\quad x\in\pa D,\\
(K\varphi)(x)&:=&\int_{\pa D}\frac{\pa\Phi_k(x,y)}{\pa\nu(y)}\varphi(y)ds(y),\quad x\in\pa D,\\
(K'\varphi)(x)&:=&\frac{\pa}{\pa\nu}\int_{\pa D}\Phi_k(x,y)\varphi(y)ds(y),\quad x\in\pa D,\\
(T\psi)(x)&:=&\frac{\pa}{\pa\nu}\int_{\pa D}\frac{\pa\Phi_k(x,y)}{\pa\nu(y)}\psi(y)ds(y),\quad x\in\pa D
\enn
and the volume integral operator $S_m:L^2(D)\to L^2(D)$ defined by
\ben
(S_m\phi)(x):=\frac{\phi(x)}{m(x)}-k^2\int_D\Phi_k(x,y)\phi(y)dy,\quad x\in D,
\enn
where $m(x)=n(x)-1\not=0$ for $x\in D$, $\nu(y)$ denotes the exterior unit normal vector at $y\in\pa D$
and $\Phi_k$ stands for the fundamental solution of the Helmholtz equation in three dimensions given by
\ben
\Phi_k(x,y)=\frac{\exp(ik|x-y|)}{4\pi|x-y|},\quad x,y\in\R^3,\quad x\not=y.
\enn
The integral operators $S,K,K':H^{-1/2}(\pa D)\to H^{1/2}(\pa D)$, $T:H^{1/2}(\pa D)\to H^{-1/2}(\pa D)$
and $S_m:L^2(D)\to L^2(D)$ are bounded operators (see \cite{CK,KG,M}).

The following factorization lemma is fundamental in the proof of the main Theorem \ref{main}.

\begin{lemma}\label{factor}
(a) For a sound-soft obstacle, we have the factorization
\be\label{f-a}
F=-4\pi GS^*G^*,
\en
where $G:H^{1/2}(\pa D)\rightarrow L^2(\Sp^2)$ is the data-to-pattern operator which maps the
Dirichlet boundary value of the radiating solution $v$ to the Helmholtz equation (\ref{he})
onto the far-field pattern $v^\infty$ and is compact, one-to-one with dense range in $L^2(\Sp^2)$.

(b) For an impedance obstacle with boundary condition ${\pa u}/{\pa\nu}+\eta u=0$ on $\pa D$ with
$\eta\in L^\infty(\pa D)$ and $\I(\eta)\ge0$, we have
\be\label{f-b}
F=-4\pi G_{imp}T^*_{imp}G^*_{imp},
\en
where $G_{imp}:H^{-1/2}(\pa D)\rightarrow L^2(\Sp^2)$ is the data-to-pattern operator which maps the
impedance boundary value of the radiating solution $v$ to the Helmholtz equation (\ref{he}) onto
the far-field pattern $v^\infty$ and is compact, one-to-one with dense range in $L^2(\Sp^2)$
and $T_{imp}:H^{1/2}(\pa D)\rightarrow H^{-1/2}(\pa D)$ is given by
\ben
T_{imp}=T+i\I(\eta) I+K'\bar\eta+\eta K+\eta S\bar\eta.
\enn

(c) For an inhomogeneous medium with real-valued index of refraction $n\in L^\infty(D)$ and
$m=n-1>0$ (or $m=n-1<0$) in $D$, we have
\be\label{f-c}
F=4\pi k^2G_mS_m^*G^*_m,
\en
where $G_m:L^2(D)\rightarrow L^2(\Sp^2)$ is the data-to-pattern operator defined by $G_mf=v^\infty$
with $v^\infty$ being the far-field pattern of the radiating solution $v$ of
\ben
\Delta v+k^2nv=-mf\quad\mbox{in}\;\;\R^3.
\enn
The data-to-pattern operator $G_m$ is injective.
\end{lemma}

\begin{proof}
(a) was proved in \cite{KG} as Theorem 1.15 (see also \cite[Theorem 3.29]{CK}), (b) was shown in \cite{KG}
as Theorem 2.6, and (c) was proved in \cite{Kirsch} as Theorem 6.28 (see also \cite[Theorem 10.12]{CK}).
\end{proof}

To proceed further we need to collect some properties of the middle operators in the factorization of
the far-field operator in Lemma \ref{factor}.

\begin{lemma}\label{middle}
Let $S_i,T_i$ be defined similarly as $S,T$, respectively with $k=i$ and let $S_0:L^2(D)\rightarrow L^2(D)$
be given by
\ben
(S_0\varphi)(x):=\frac{\varphi(x)}{m(x)},\quad x\in D.
\enn
Then the following statements are true:

$(1)$ $S_i$ is self-adjoint with respect to $L^2(\pa D)$ and coercive.

$(2)$ $S-S_i:H^{-1/2}(\pa D)\rightarrow H^{1/2}(\pa D)$ is compact.

$(3)$ $-T_i$ is self-adjoint with respect to $L^2(\pa D)$ and coercive.

$(4)$ $T-T_i$ is compact from $H^{1/2}(\pa D)$ to $H^{-1/2}(\pa D)$.

$(5)$ $T-T_i+K'\bar\eta+\eta K+\eta S\bar\eta:H^{1/2}(\pa D)\rightarrow H^{-1/2}(\pa D)$ is compact.

$(6)$ $S_0$ is coercive for the case $m(x)=n(x)-1\ge c_1$ in $D$ $(-S_0$ is coercive for the case
$m(x)=n(x)-1\le-c_1$ in $D)$ for some constant $c_1>0$.

$(7)$ $S_m-S_0:L^2(D)\rightarrow L^2(D)$ is compact.
\end{lemma}

\begin{proof}
Statements (1) and (2) are proved in \cite{CK,KG} (see Lemma 5.37 in \cite{CK} or Lemma 1.14 (c) and (d)
in \cite{KG}).

Statements (3) and (4) are proved in \cite{KG} (see Theorem 1.26 (e) and (f) in \cite{KG},
where different notations $N,N_i$ are used for $T,T_i$, respectively).

Statement (5) follows easily from (4) and the fact that $S,K,K'$ are compact from
$H^{1/2}(\pa D)$ to $H^{-1/2}(\pa D)$.

Statements (6) and (7) are shown in \cite{CK,Kirsch} for the case $m=n-1>0$ in $D$ (see Theorem 10.14
in \cite{CK} or Theorem 6.30 in \cite{Kirsch}). The case $m=n-1<0$ in $D$ can be shown similarly.
\end{proof}

By Lemma \ref{normal} the far-field operator $F$ is normal for the cases considered in this paper, so
$F$ has a countably infinite number of eigenvalues. However, we need to prove that $F$ has a countably
infinite number of non-zero eigenvalues. To this end, we have to show that $F$ has finite dimensional
null space.

\begin{lemma}\label{nullspace}
For a sound-soft obstacle or an impedance obstacle with real-valued impedance function
$\eta\in L^\infty(\pa D)$ and for an inhomogeneous medium with real-valued index of refraction
$n\in L^\infty(\R^3)$ satisfying that $n=1$ in $\R^3\se{D}$ and $n>1$ $(\mbox{or}\;\;n<1)$ in $D$,
the dimension of the null space of $F$ is finite.
\end{lemma}

\begin{proof}
For the inhomogeneous medium case, the result was proved in \cite{CK} (see Theorem 8.16 in \cite{CK}).

We only give a proof for the impedance obstacle case since the proof for a sound-soft obstacle
is similar. Our proof follows a similar idea as in the proof of Lemma 2 in \cite{LP}.

Note first that if $Fg=0$ for $g\in L^2(\Sp^2)$ then the Herglotz wave function $v_g$ defined by
\ben
v_g(x)=\int_{\Sp^2}e^{ikx\cdot d}g(d)ds(d),\quad x\in\R^3
\enn
is an eigenfunction of the negative impedance-Laplacian in $D$ corresponding to the eigenvalue $k^2$.
Thus, and by the one-to-one correspondence between Herglotz wave functions and their kernels $g$ (see
\cite[Theorem 3.19]{CK}), the finiteness of the dimension of the null space of $F$ follows from that
of the eigenspace of the negative impedance-Laplacian in $D$ associated with the eigenvalue $k^2$.

We claim that the dimension of the eigenspace of the negative impedance-Laplacian in $D$ associated
with the eigenvalue $k^2$ is equal to the dimension of the kernel space of $T+K'\eta+\eta K+\eta S\eta$.
Now, by Lemma \ref{middle} we know that
\[T+K'\eta+\eta K+\eta S\eta=T_i+(T-T_i)+K'\eta+\eta K+\eta S\eta\]
is a Fredholm-type operator, so the dimension of the kernel space of $T+K'\eta+\eta K+\eta S\eta$ is
finite. Consequently, the dimension of the null space of $F$ is finite.

In fact, if $u$ is an eigenfunction of the negative impedance-Laplacian in $D$ corresponding to the
eigenvalue $k^2$, then from the Green's representation formula it follows that
\ben
u(x)&=&\int_{\pa D}\left\{\frac{\pa u}{\pa\nu}(y)\Phi_k(x,y)-u(y)\frac{\pa\Phi_k(x,y)}{\pa\nu(y)}\right\}ds(y)\\
&=&-\int_{\pa D}\left(\eta(y)\Phi_k(x,y)+\frac{\pa\Phi_k(x,y)}{\pa\nu(y)}\right)u(y)ds(y),\quad x\in D.
\enn
By the jump relations of the single and double layer potentials, we get
\ben
2u&=&u-2\int_{\pa D}\left(\frac{\pa\Phi_k(x,y)}{\pa\nu(y)}+\eta(y)\Phi_k(x,y)\right)u(y)ds(y),\quad x\in\pa D,\\
2\frac{\pa u}{\pa\nu}&=&-\eta u-2\frac{\pa}{\pa\nu}\int_{\pa D}\left(\frac{\pa\Phi_k(x,y)}{\pa\nu(y)}
+\eta(y)\Phi_k(x,y)\right)u(y)ds(y),\quad x\in\pa D.
\enn
Therefore, we have
\ben
0=2\left(\frac{\pa u}{\pa\nu}+\eta u\right)=-2\left(T+K'\eta+\eta K+\eta S\eta\right)u,
\enn
that is, $u$ is an element in the kernel space of $T+K'\eta+\eta K+\eta S\eta$.
Conversely, let $(T+K'\eta+\eta K+\eta S\eta)\varphi=0$ for $\varphi\in H^{1/2}(\pa D)$.
Define the combined double- and single-layer potential:
\ben
u(x):=\int_{\pa D}\left(\frac{\pa\Phi_k(x,y)}{\pa\nu(y)}+\eta(y)\Phi_k(x,y)\right)\varphi(y)ds(y),
\quad x\in\R^3\setminus\pa D.
\enn
By the jump relations of the layer potentials, it is easy to show that $u$ is an eigenfunction of
the negative impedance-Laplacian in $D$ corresponding to the eigenvalue $k^2$.
This means that the required claim is correct.
The proof is thus complete.
\end{proof}

Combining Lemmas \ref{normal} and \ref{nullspace} we can establish the following theorem.

\begin{theorem}\label{infinite}
For a sound-soft obstacle or an impedance obstacle with real-valued impedance function
$\eta\in L^\infty(\pa D)$ and for an inhomogeneous medium with real-valued index of refraction
$n\in L^\infty(\R^3)$ satisfying that $n=1$ in $\R^3\se{D}$ and $n>1$ $(\mbox{or}\;\;n<1)$ in $D$,
the far field operator $F$ has a countably infinite number of eigenvalues accumulating only at $0$,
and the multiplicity of each eigenvalue is finite. Furthermore, the eigenfunctions form an orthonormal
basis of $L^2(\Sp^2)$.
\end{theorem}

\begin{proof}
From Lemma \ref{normal} and the spectral theorem for normal operators, we know that $F$ has a
countably infinite number of eigenvalues $\{\lambda_n\}_{n=1}^\infty$ with the corresponding
eigenfunctions $\{g_n\}_{n=1}^\infty$ forming an orthonormal basis of $L^2(\Sp^2)$.
We now prove that the multiplicity of each eigenvalue is finite.
Suppose, to the contrary, that the multiplicity of $\lambda_n$ is infinite for some $n\in\Z$,
that is, $\mbox{dim}[\mbox{Ker}(\lambda_n I-F)]=\infty$.
Then the compactness of $F$ implies that $\lambda_n=0$. But this is impossible
by Lemma \ref{nullspace}. From the compactness of the far field operator $F$ again, the
unique accumulation point of the eigenvalues $\{\lambda_n\}_{n=1}^\infty$ is $0$.
The proof is thus complete.
\end{proof}

\subsection{Limits of the normalized eigenvalues}

Eckmann and Pillet \cite{EP} first established the limit of the normalized eigenvalues of
the far-field operators for a piecewise smooth sound-soft obstacle in $\R^2$
whether or not $k^2$ is an interior Dirichlet eigenvalue of the obstacle,
in the context of quantum billiards, by using a variational principle.
Similar results are given in the monograph \cite{KG} for sound-soft or sound-hard obstacles in $\R^3$
in order to characterize the obstacle by using the eigensystem of the far-field operator
when $k^2$ is not an eigenvalue of the underlying boundary value problem in the obstacle
(see the proof of Theorem 1.23 of \cite{KG}).
Similar results are also presented in \cite{Kirsch} for inhomogeneous media with real-valued indices
of refraction in $\R^3$ in the case when $k^2$ is not an eigenvalue of the underlying interior
transmission problem in the support of the contrast of the inhomogeneous medium (see Lemma 6.34
of \cite{Kirsch}).
In this subsection, we extend these results to both the case of sound-soft or impedance obstacles
with real-valued impedance coefficients and the case with an inhomogeneous medium with real-valued
index of refraction no matter whether or not $k^2$ is an eigenvalue of the underlying boundary value
problem in the obstacle or the underlying interior transmission problem in the support of the contrast
of the inhomogeneous medium. To this end, we give the following general result on the property of the
eigenvalues of a linear compact normal operator having a special factorization form.

\begin{theorem}\label{tend1}
Let $H$ and $X$ be Hilbert spaces with inner products $(\cdot,\cdot)$, let $X^*$ be the dual space
of $X$ and assume that $F:H\rightarrow H$ is a linear compact normal operator satisfying
\[F=GM^*G^*,\]
where $G:X\rightarrow H$ and $M:X^*\rightarrow X$ are bounded linear operators and
$G^*: H\rightarrow X^*$ is the adjoint of $G$ defined by
\[\langle G^*g,\varphi\rangle=(g,G\varphi),\quad\forall\varphi\in X,g\in H,\]
in terms of the sesqui-linear duality pairing of $X^*$ and $X$.
Assume further that $G^*$ has a finite dimensional null space and $M=M_0+C$ for some compact
operator $C$ and some self-adjoint operator $M_0$ which is coercive on $G^*(H)$ in the sense that
there exists $c_0>0$ such that
\ben
\langle\phi,M_0\phi\rangle\ge c_0\|\phi\|^2\qquad\mbox{for all}\;\;\phi\in G^*(H).
\enn
Then $F$ has at most only a finite number of eigenvalues whose real parts are negative.
\end{theorem}

\begin{proof}
Suppose, to the contrary, that $F$ has a countably infinite number of eigenvalues whose real
parts are negative. By the spectral theorem for compact normal operators, there exists
an infinite number of orthonormal eigenelements $g_n\in H$ with corresponding eigenvalues
$\lambda_n$, $n\in\N$. Then we can choose a subsequence of $\{g_{n}\}_{n=1}^\infty$, which
we denote by $\{g_{n}\}_{n=1}^\infty$ again, such that $\Rt(\la_{n})<0$ for all $n$.
We may assume without loss of generality that $G^*(g_{n})\not=0$ since the dimension of the
null space of $G^*$ is finite. Define $f_{n}:=G^*g_{n}\in X^*$. Then
\ben
\Rt(Fg_{n},g_{n})&=&\Rt(GM^*G^*g_{n},g_{n})=\Rt\langle G^*g_{n},MG^*g_{n}\rangle\\
&=&\Rt\langle f_{n},Mf_{n}\rangle=\Rt\langle f_{n},M_0f_{n}\rangle +\Rt\langle
f_{n},Cf_{n}\rangle.
\enn
Thus, and since $\Rt(Fg_{n},g_{n})=\Rt(\lambda_{n})<0$, we obtain that
\ben
c_0\|f_{n}\|^2\leq\Rt\langle f_{n},M_0f_{n}\rangle<-\Rt\langle f_{n},Cf_{n}\rangle
\leq\left|\langle f_{n},Cf_{n}\rangle\right|\leq\|f_{n}\|\|Cf_{n}\|.
\enn
Since $f_{n}=G^*g_{n}\not=0$, it follows that
\ben
c_0\|f_{n}\|<\|Cf_{n}\|\qquad\mbox{for}\;\;\mbox{all}\;\;n\in\N.
\enn
Noting that $(Fg_{n},g_{k})=\lambda_{n}\delta_{nk}$, we have
\ben
&&\Rt(F(c_{n}g_{n}+c_{k}g_{k}),c_{n}g_{n}+c_{k}g_{k})\\
&&\qquad\qquad=|c_{n}|^2\Rt(Fg_{n},g_{n})+|c_{k}|^2\Rt(Fg_{k},g_{k})<0,
\enn
where $n\not=k$ and $c_{n},c_{k}\in\C$ are arbitrary constants satisfying that
$|c_{n}|+|c_{k}|\not=0$. Again, since the dimension of the null space of $G^*$ is finite,
for $f_{n_1}:=f_{1}=G^*g_{1}$ we can choose some element from $\{g_{n}\}_{n=2}^\infty$,
which we denote by $g_{n_2}$, such that
$$
c_{n_1}f_{n_1}+c_{n_2}f_{n_2}=c_{n_1}G^*g_{1}+c_{n_2}G^*g_{n_2}\not=0
$$
for any constants $c_{n_1},c_{n_2}\in\C$ with $|c_{n_k}|+|c_{n_l}|\not=0$, that is,
$f_{n_1}$ and $f_{n_2}$ are linearly independent.
Noting that $\Rt(F(c_{n_1}g_{n_1}+c_{n_2}g_{n_2}),c_{n_1}g_{n_1}+c_{n_2}g_{n_2})<0$, and
by the same argument as above, it follows that
\ben
c_0\|c_{n_1}f_{n_1}+c_{n_2}f_{n_2}\|<\|C(c_{n_1}f_{n_1}+c_{n_2}f_{n_2})\|.
\enn
Repeating the above process, we can choose $g_{n_k}$, $k=1,\cdots,N$ such that
\ben
\sum_{k=1}^Nc_{n_k}f_{n_k}\not=0,\qquad
c_0\left\|\sum_{k=1}^Nc_{n_k}f_{n_k}\right\|<\left\|C\sum_{k=1}^Nc_{n_k}f_{n_k}\right\|
\enn
for any $N<\infty$ and arbitrary constants $c_{n_k}\in\C$, $k=1,\cdots,N$, with
$\sum_{k=1}^N|c_{n_k}|\not=0$. Note that $\{f_{n_k}\}_{k=1}^N=\{G^*g_{n_k}\}_{k=1}^N$ are
linearly independent. Then, by the Gram-Schmidt orthonormalization process
(see \cite[pp. 110]{DM}) there exists an orthonormal system $\{e_{n_k}\}_{k=1}^\infty\subset X^*$
such that
\ben
\mbox{span}\{f_{n_1},...,f_{n_k}\}=\mbox{span}\{e_{n_1},...,e_{n_k}\},\quad\forall k\in\N
\enn
and
\ben
c_0=c_0\|e_{n_k}\|<\|Ce_{n_k}\|.
\enn
Since $e_{n_k}$ is weakly convergent to zero in $X^*$ as $k\to\infty$ (see \cite[pp. 112]{DM}),
and by the compactness of $C$, we have
\ben
c_0=c_0\|e_{n_k}\|<\|Ce_{n_k}\|\rightarrow0.
\enn
This is a contradiction. The proof is thus complete.
\end{proof}

Making use of Theorem \ref{tend1} in conjunction with Lemmas \ref{factor} and \ref{middle} and
Theorem \ref{infinite}, we can prove the following lemma.

\begin{lemma}\label{one-way}
Suppose $\{(\lambda_n,g_n)\}_{n=1}^\infty$ is the eigensystem of the far field operator $F$
such that $\{g_n\}_{n=1}^\infty$ forms an orthonormal basis of $L^2(\Sp^2)$. Then we have

(a) In the sound-soft case, there only exist finitely many eigenvalues $\lambda_n$ satisfying
that $\Rt(\lambda_n)>0$;

(b) In the non-absorbing impedance case, there only exist finitely many eigenvalues $\lambda_n$
satisfying that $\Rt(\lambda_n)<0$;

(c) In the non-absorbing medium case, there only exist finitely many eigenvalues $\lambda_n$
satisfying that
\ben
\Rt(\lambda_n)\left\{\begin{array}{ll}
>0 &\mbox{if}\;\; m\le-c_1\;\;\mbox{in}\;\;D,\\
<0 & \mbox{if}\;\;m\ge c_1\;\;\mbox{in}\;\;D
\end{array}\right.
\enn
for some constant $c_1>0$.
\end{lemma}

\begin{proof}
We only prove (c). The proof of (a) and (b) is similar.

We now apply Theorem \ref{tend1} to prove (c). To do this, let $H=L^2(\Sp^2)$, $X=L^2(D)$,
$G=G_m$, $M=4\pi k^2S_m$ with $M_0=4\pi k^2S_0$ and $C=4\pi k^2(S_m-S_0)$. Then $F=GM^*G^*$.
Now, by Lemma \ref{middle} (6) and (7) and Lemma \ref{factor} (c) it is known that $M$
satisfies the assumption in Theorem \ref{tend1} for the case $m(x)\ge c_1$ in $D$.
Further, by Theorem 6.30 (c), $S_m$ or equivalently $M$ is an isomorphism from $L^2(D)$ onto
itself. This, together with the fact that the far-field operator $F$ is compact and normal
(by Lemma \ref{normal}) and has finite dimensional null space (by Lemma \ref{nullspace}) and
$G$ is injective, implies that $G^*$ has finite dimensional null space.
Thus all the assumptions in Theorem \ref{tend1} are satisfied so, by Theorem \ref{tend1}
we have that the far-field operator $F$ has at most only a finite number of eigenvalues
$\lambda_n$ with $\Rt(\lambda_n)<0$.

For the case when $m(x)\le-c_1$ in $D$, $-F=G(-M)^*G^*$ and $-M=-4\pi k^2S_m$ with
$-M_0=-4\pi k^2S_0$ being coercive and $-C=-4\pi k^2(S_m-S_0)$ being compact.
Thus, by Theorem \ref{tend1}, $-F$ has at most only a finite number of eigenvalues $-\lambda_n$ with
$-\Rt(\lambda_n)<0$, or equivalently, $F$ has at most only a finite number of eigenvalues
$\lambda_n$ with $\Rt(\lambda_n)>0$. The proof is thus complete.
\end{proof}

We are now in a position to state and prove the main theorem of this section.

\begin{theorem}\label{eigentend}
Let $\{\lambda_n\}_{n=1}^\infty$ be the eigenvalues of the far-field operator $F$. Then
\ben
\lim_{n\rightarrow\infty}\frac{\lambda_n}{|\lambda_n|}=\left\{\begin{array}{ll}
-1 & \text{for a sound-soft obstacle},\\
1 & \text{for a non-absorbing impedance obstacle},\\
\text{sign}(m) & \text{for a non-absorbing medium with $|m|\ge c_1>0$ in $D$}.
\end{array}\right.
\enn
\end{theorem}

\begin{proof}
By Remark \ref{re-circle} and Theorem \ref{infinite} we know that the possible accumulations
of the normalized non-zero eigenvalues of $F$ can only be $\pm 1$.
Theorem \ref{eigentend} then follows easily from Theorem \ref{infinite}
and Lemma \ref{one-way}.
The proof is complete.
\end{proof}

\section{Proofs of the main results}\label{sec4}

In this section we prove our main results, Theorems \ref{main} and \ref{main2}.

{\em Proof of Theorems $\ref{main}$.}
From (\ref{lsp}) it is easy to see that (\ref{m=}) is equivalent to the equation
\be\label{modulus=}
|u_1^\infty(\hat{x},d_1)+u_1^\infty(\hat{x},d_2)|=|u_2^\infty(\hat{x},d_1)+u_2^\infty(\hat{x},d_2)|,\quad
\forall\hat{x},d_1,d_2\in\Sp^2.
\en
This implies that
\be\label{interterm}
2\Rt\{u_1^\infty(\hat{x},d_1)\ov{u_1^\infty(\hat{x},d_2)}\}
=2\Rt\{u_2^\infty(\hat{x},d_1)\ov{u_2^\infty(\hat{x},d_2)}\}.
\en
Define $r_j(\hat{x},d):=|u_j^\infty(\hat{x},d)|$, $j=1,2$. Then, by (\ref{modulus=}) with $d_1=d_2=:d$
we have
\be\label{thm1-1}
r_1(\hat{x},d)=r_2(\hat{x},d)=:r(\hat{x},d),\quad\forall\hat{x},d\in\Sp^2,
\en
so
\ben
u_j^\infty(\hat{x},d)=r(\hat{x},d)e^{i\theta_j(\hat{x},d)},\quad\forall\hat{x},d\in\Sp^2,\;\;j=1,2,
\enn
where $\theta_j(\hat{x},d)$, $j=1,2$, are real-valued functions of both variables.
For the case $r(\hat{x},d)\equiv0$, we have $u_1^\infty(\hat{x},d)\equiv u_2^\infty(\hat{x},d)\equiv0$
for all $\hat{x},d\in\Sp^2$.

We now consider the case $r(\hat{x},d)\not\equiv0$ for $\hat{x},d\in\Sp^2$.
Define
\ben
U:=\{(\hat{x},d)\in\Sp^2\times\Sp^2|r(\hat x,d)\not=0\}.
\enn
Then, by the continuity of $r(\hat{x},d)$ in $(\hat{x},d)\in\Sp^2\times\Sp^2$ it follows that
$U$ is an open domain of $\Sp^2\times\Sp^2$. Further, since $u_j^\infty(\hat{x},d)$, $j=1,2$, are
analytic functions of $\hat{x}\in\Sp^2$ and $d\in\Sp^2$, respectively, we may choose two open
sets $U_1,U_2\subset\Sp^2$ small enough so that $U_1\times U_2\subset U$ and $\theta_j(\hat{x},d)$,
$j=1,2$, are analytic with respect to $\hat{x}\in U_1$ and $d\in U_2$, respectively.

Now, by (\ref{interterm}) and the definition of $U$ we have
\be\label{cos=}
\cos[\theta_1(\hat{x},d_1)-\theta_1(\hat{x},d_2)]=\cos[\theta_2(\hat{x},d_1)-\theta_2(\hat{x},d_2)]
\en
for all $(\hat{x},d_j)\in U_1\times U_2$, $j=1,2$. By (\ref{cos=}) and the fact that
$\theta_j(\hat{x},d)$, $j=1,2$, are real-valued analytic functions of $\hat{x}\in U_1$ and $d\in U_2$,
respectively, we obtain that there holds either
\be\label{case1}
\theta_1(\hat{x},d_1)-\theta_1(\hat{x},d_2)=\theta_2(\hat{x},d_1)-\theta_2(\hat{x},d_2),\;\;
\forall(\hat{x},d_j)\in U_1\times U_2
\en
or
\be\label{case2}
\theta_1(\hat{x},d_1)-\theta_1(\hat{x},d_2)=-[\theta_2(\hat{x},d_1)-\theta_2(\hat{x},d_2)],\;\;
\forall(\hat{x},d_j)\in U_1\times U_2,
\en
where $j=1,2.$

For the case when (\ref{case1}) holds, we have
\ben
\theta_1(\hat{x},d_1)-\theta_2(\hat{x},d_1)=\theta_1(\hat{x},d_2)-\theta_2(\hat{x},d_2),\;\;
\forall(\hat{x},d_j)\in U_1\times U_2,\;\;j=1,2.
\enn
Fix $d_2\in U_2$ and define
\ben
\alpha(\hat{x}):=\theta_1(\hat{x},d_2)-\theta_2(\hat{x},d_2),\quad\forall\hat{x}\in U_1.
\enn
Then, by (\ref{case1}) we have
\ben
u_1^\infty(\hat{x},d)=r(\hat{x},d)e^{i\theta_1(\hat{x},d)}
=r(\hat{x},d)e^{i\alpha(\hat{x})+i\theta_2(\hat{x},d)}=e^{i\alpha(\hat{x})}u_2^\infty(\hat{x},d)
\enn
for all $(\hat{x},d)\in U_1\times U_2$, where we use $d$ to replace $d_1$. By the analyticity of
$u_1^\infty(\hat{x},d)-e^{i\alpha(\hat{x})}u_2^\infty(\hat{x},d)$ in $d\in\Sp^2$, we get
\be\label{case1-1}
u_1^\infty(\hat{x},d)=e^{i\alpha(\hat{x})}u_2^\infty(\hat{x},d),\quad\forall\hat{x}\in U_1,d\in\Sp^2.
\en
Changing the variables $\hat{x}\rightarrow -d$ and $d\rightarrow-\hat{x}$ in (\ref{case1-1}) gives
\ben
u_1^\infty(-d,-\hat{x})=e^{i\alpha(-d)}u_2^\infty(-d,-\hat{x}),\quad\forall -d\in U_1,\hat{x}\in\Sp^2.
\enn
The reciprocity relation $u_j^\infty(\hat{x},d)=u_j^\infty(-d,-\hat{x})$ for all $\hat{x},d\in\Sp^2$
($j=1,2$) leads to the result
\be\label{eq1}
e^{i\alpha(\hat{x})}u^\infty_2(\hat{x},d)=e^{i\alpha(-d)}u^\infty_2(\hat{x},d),\quad\forall\hat{x},-d\in U_1.
\en
Since $r(\hat{x},d)\not\equiv0$ for $\hat{x},d\in\Sp^2$, it follows from (\ref{eq1}) and
the analyticity of $u^\infty_j(\hat{x},d)$ ($j=1,2$) with respect to $\hat{x}\in\Sp^2$ and $d\in\Sp^2$,
respectively, that
\ben
\exp[{i\alpha(\hat{x})}]=\exp[{i\alpha(-d)}],\quad \forall \hat{x},-d\in U_1.
\enn
For a fixed $-\wid{d}\in U_1$ take $d=\wid{d}$ in the above formula to give that $\exp[{i\alpha(\hat{x})}]
=\exp({i\alpha})$ for all $\hat{x}\in U_1$, where $\alpha:=\alpha(-\wid{d})$ is a real constant.
Substituting this formula into (\ref{case1-1}) gives
\ben
u_1^\infty(\hat{x},d)=e^{i\alpha}u_2^\infty(\hat{x},d),\quad\forall\hat{x}\in U_1,d\in\Sp^2.
\enn
By the analyticity of $u_j^\infty(\hat{x},d)$ ($j=1,2$) with respect to $\hat{x}\in\Sp^2$ it follows that
\be\label{Case1-result}
u_1^\infty(\hat{x},d)=e^{i\alpha}u_2^\infty(\hat{x},d),\quad\forall\hat{x},d\in\Sp^2.
\en

For the case when (\ref{case2}) holds, a similar argument as above gives the result
\be\label{Case2-result}
u_1^\infty(\hat{x},d)=e^{i\beta}\ov{u_2^\infty(\hat{x},d)},\quad\forall\hat{x},d\in\Sp^2,
\en
where $\beta$ is a real constant.

We now prove that (\ref{Case2-result}) does not hold. In fact, suppose $\{(\lambda_n,g_n)\}_{n=1}^\infty$
is the eigensystem of the far-field operator $F_1$ corresponding to the obstacle $D_1$ or the refraction
of index $n_1$. Then, by (\ref{Case2-result}) we have
\ben
(F_2\ov{g_n})(\hat{x})&=&\int_{\Sp^2}u^\infty_2(\hat{x},d)\ov{g_n(d)}ds(d)
=\int_{\Sp^2}e^{i\beta}\ov{u^\infty_1(\hat{x},d)g_n(d)}ds(d)\\
&=&e^{i\beta}\ov{(F_1g_n)(\hat{x})}=e^{i\beta}\ov{\lambda_n}\ov{g_n(\hat{x})}.
\enn
Thus, $\{(e^{i\beta}\ov{\lambda_n},\ov{g_n})\}_{n=1}^\infty$ is the eigensystem of the far-field operator $F_2$
corresponding to the obstacle $D_2$ or the refraction of index $n_2$. Theorem \ref{eigentend} implies that
\be\label{Case2-F1}
\lim_{n\rightarrow\infty}\frac{\lambda_n}{|\lambda_n|}
&=&\left\{\begin{array}{ll}
-1 & \text{if $D_1$ is a sound-soft obstacle},\\
1 & \text{if $D_1$ is a non-absorbing impedance obstacle},\\
\text{sign}(m_1) & \text{for real-valued $n_1$ with $|m_1|\ge c_1$ in $D_1$},
\end{array}\right.\\ \label{Case2-F2}
\lim_{n\rightarrow\infty}\frac{e^{i\beta}\ov{\lambda_n}}{|e^{i\beta}\ov{\lambda_n}|}
&=&\left\{\begin{array}{ll}
-1 & \text{if $D_2$ is a sound-soft obstacle},\\
1 & \text{if $D_2$ is a non-absorbing impedance obstacle},\\
\text{sign}(m_2) & \text{for real-valued $n_2$ with $|m_2|\ge c_1$ in $D_2$},
\end{array}\right.\;\;
\en
where $m_j=n_j-1$ in $D_j$, $j=1,2$.
Note that $D_1$ and $D_2$ are both either sound-soft obstacles or
non-absorbing impedance obstacles. Note further that $n_1$ and $n_2$ are both real-valued and satisfy
that either $m_j=n_j-1\ge c_1$ in $D_j$ ($j=1,2$) or $m_j=n_j-1\le-c_1$ in $D_j$ ($j=1,2$),
so the sign of $m_1$ in $D_1$ is the same as that of $m_2$ in $D_2$.
Therefore, from (\ref{Case2-F1}) and (\ref{Case2-F2}) it follows that $e^{i\beta}=1$.
Hence $\{\ov{\lambda_n}\}_{n=1}^\infty$ are the eigenvalues of the far-field operator $F_2$.
Recall that $\{\lambda_n\}_{n=1}^\infty$ are the eigenvalues of the far-field operator $F_1$.
By Remark \ref{re-circle}, both non-zero elements of $\{\lambda_n\}_{n=1}^\infty$ and
$\{\ov{\lambda_n}\}_{n=1}^\infty$ lie on the circle (\ref{circle}) in the upper-half complex plane,
which is a contradiction. This means that (\ref{Case2-result}) does not hold.

We now consider (\ref{Case1-result}). Let $\{(\lambda_n,g_n)\}_{n=1}^\infty$ be the eigensystem of
the far-field operator $F_1$ corresponding to the obstacle $D_1$ or the refraction of index $n_1$.
Then it follows from (\ref{Case1-result}) that
\ben
(F_2g_n)(\hat x)&=&\int_{\Sp^2}u^\infty_2(\hat{x},d)g_n(d)ds(d)
=\int_{\Sp^2}e^{-i\alpha}u^\infty_1(\hat{x},d)g_n(d)ds(d)\\
&=&e^{-i\alpha}(F_1g_n)(\hat{x})=e^{-i\alpha}\lambda_ng_n(\hat{x}).
\enn
Thus, $\{(e^{-i\alpha}\lambda_n,g_n)\}_{n=1}^\infty$ is the eigensystem of the far-field operator $F_2$
corresponding to the obstacle $D_2$ or the refraction of index $n_2$. By Theorem \ref{eigentend}, we have
\ben\label{Case1-F1}
\lim_{n\rightarrow\infty}\frac{\lambda_n}{|\lambda_n|}
&=&\left\{\begin{array}{ll}
-1 & \text{if $D_1$ is a sound-soft obstacle},\\
1 & \text{if $D_1$ is a non-absorbing impedance obstacle},\\
\text{sign}(m_1) & \text{for real-valued $n_1$ with $|m_1|\ge c_1$ in $D_1$},
\end{array}\right.\\ \label{Case1-F2}
\lim_{n\rightarrow\infty}\frac{e^{-i\alpha}\ov{\lambda_n}}{|e^{-i\alpha}\ov{\lambda_n}|}
&=&\left\{\begin{array}{ll}
-1 & \text{if $D_2$ is a sound-soft obstacle},\\
1 & \text{if $D_2$ is a non-absorbing impedance obstacle},\\
\text{sign}(m_2) & \text{for real-valued $n_2$ with $|m_2|\ge c_1$ in $D_2$},
\end{array}\right.\;\;
\enn
Similarly as above, we obtain that $e^{-i\alpha}=1$. Hence
\be\label{equiv}
u_1^\infty(\hat{x},d)=u_2^\infty(\hat{x},d),\;\;\forall \hat{x},d\in\Sp^2.
\en

For two sound-soft or non-absorbing impedance obstacles $D_1$ and $D_2$, by (\ref{equiv}) and
\cite[Theorem 5.6]{CK} we have $D_1=D_2$ and $\eta_1=\eta_2$.
For two indices of refraction $n_1$ and $n_2$ satisfying the assumptions in Theorem \ref{main},
by (\ref{equiv}) and \cite[Theorem 6.26]{Kirsch} we have $n_1=n_2$. Theorem \ref{main} is thus proved.
$\Box$

{\em Proof of Theorem $\ref{main2}$.}
By (\ref{r=}) and (\ref{m=2}) we know that the equations (\ref{modulus=}) and (\ref{interterm}) still
hold with $d_1=d$ and $d_2=d_0$. Further, by (\ref{r=}) the equation (\ref{thm1-1}) holds.
Then the remaining part of the proof of Theorem \ref{main} works by letting $d_1=d$ and $d_2=d_0$
and noting that $(\hat{x},d_0)\in U$ for all $\hat{x}$ in a small open domain of $\Sp^2$ (since,
otherwise, $u^\infty_j(\hat{x},d_0)\equiv0$ for all $\hat{x}\in\Sp^2$ so, by Rellich's lemma,
$u^s_j(x,d_0)\equiv0$ for all $x\in\R^3\se\ov{D_j}$, leading to the contradiction that $u^i(x,d_0)\equiv0$
or $\pa u^i(x,d_0)/\pa\nu+\eta u^i(x,d_0)\equiv0$ for all $x\in\pa D_j$, $j=1,2$).
The proof is complete. $\Box$

\section{Conclusion}\label{conclusion}

In this paper, we established uniqueness in inverse acoustic scattering with phaseless far-field patterns
associated with infinitely many sets of superpositions of two plane waves with different
directions at a fixed frequency for sound-soft or non-absorbing impedance obstacles and for a non-absorbing
inhomogeneous medium with the real-valued index of refraction $n$ satisfying that either $n(x)-1\ge c_1$
or $n(x)-1\le-c_1$ in the support $D$ for some constant $c_1>0$.
To the best of our knowledge, this is the first uniqueness result in inverse scattering with phaseless
far-field data. As an ongoing project, we are currently trying to remove the non-absorbing condition on
the boundary impedance function $\eta$ and the index of refraction $n$.
In the near future, we hope to extend the results to the case without a priori knowing the property of
the obstacle and the refractive index of the inhomogeneous medium and to the case of electromagnetic waves.

\section*{Acknowledgements}

This work is partly supported by the NNSF of China grants 91430102, 91630309, 61379093 and 11501558
and the National Center for Mathematics and Interdisciplinary Sciences, CAS.

\end{document}